\documentclass[11pt]{article}
\usepackage{amssymb}
\usepackage{amsmath}
\usepackage{amsthm}
\usepackage[round]{natbib}
\usepackage{graphicx}
\usepackage{wasysym}
\usepackage{amsfonts}
\usepackage{multirow}
\usepackage{paralist}
\usepackage{verbatim}
\usepackage{bm}
\usepackage{MnSymbol}
\usepackage{subfig}
\usepackage{rotating}
\usepackage{float}
\restylefloat{table}

\newtheorem{lemma}{Lemma}
\newtheorem{theorem}{Theorem}
\begin{document}

\title{On Transshipment Games with Identical Newsvendors}


\author{Behzad ˜Hezarkhani\\
OPAC, Eindhoven University of Technology, Eindhoven, The Netherlands\\
b.hezarkhani@tue.nl\\
\and
Wieslaw ˜Kubiak\\
Faculty of Business Administration, Memorial University, St.\ John's, Canada \\
wkubiak@mun.ca\\
\and
Bruce Hartman \\
Kogod School of Business, American University, Washington, DC USA \\
bhartman@stfrancis.edu}
\date{\today }
\maketitle

\begin{abstract}
In a transshipment game, supply chain agents cooperate
to transship  surplus products. This note studies the effect of size of transshipment coalitions on the optimal production/order
quantities. It characterizes these quantities for transshipment
games with identical newsvendors and normally distributed market demands. It also gives
a closed form formula for equal allocation in their cores.
\end{abstract}


\section{Introduction}

Transshipment is the practice of sharing common resources among supply chain
entities who face uncertain market demands. 
Although transshipment games have been extensively studied in the literature (see %
\citet{Paterson2009} for a recent review), their complexity hinders
the derivation of straightforward analytical results in general. Specifically, the
effect of size of transshipping coalition, i.e.\ the number of transshipping
locations, on the production/order quantities has never been investigated
before in the literature. 
The characterization of optimal quantities is useful in both centralized
supply chain contexts (e.g.\ \citet{Herer2006}), and cooperative
transshipment games (\citet{Hezarkhani2010a}, \citet{Anupindi2001}, and %
\citet{Sosic2006} among others). In the latter case, the key result of %
\citet{Slikker2005} ensures a non-empty core for a transshipment game. 
 However, the question of
how the growth of transshipping coalitions affect the optimal quantities and
expected profit remains open and it will be addressed in this note.

In Section 3 and 4, we characterize the main properties of transshipment amounts and optimal quantities respectively in multi-location/multi-agent
transshipment games for identical agents facing normally distributed
demands. The games are introduced in Section 2. There are three categories of transshipment games:
over-mean, under-mean, and mean games. The game category depends on
the optimal quantity, i.e.\ critical fractile, of a single newsvendor
entity. As the game size grows these optimal quantities get closer to the
distribution mean for the over- and under-mean problems. However, for either
category, we show that there is a threshold value for the transportation cost $t$ such
that the optimal quantity converges to the demand distribution mean for
transportation cost not exceeding the threshold, and to a certain bound
different from the mean for transportation costs above the
threshold. A closed form formula for equal allocation in the core is derived in Section 5.

\section{Cooperative Transshipment Games}

Consider a set $N$ of $n$ newsvendor agents. Each agent $i\in N$ decides its production/order quantity (simply quantity hereafter), $X_{i}$, in anticipation of a random demand $D_{i}$ having continuous and twice differentiable pdf with mean $\mu_{i}$ and standard deviation $\sigma_{i}$. The market selling price, purchasing cost, and salvage value are denoted by $r_{i}$, $c_{i}$, and $\nu_{i}$ respectively ($\nu _{i}<c_{i}<r_{i}$). The newsvendors have the option to form a transshipment coalition to transship their otherwise surplus products to other members of the coalition after the realization of demands. In order to physically move one unit of product from newsvendor $i$ to newsvendor $j$, both members of the same coalition, the transportation cost $t_{ij}$ is incurred. 
 The $W_{ij}$ is the  quantity transshipped from newsvendor $i$ to $j$. In order to avoid trivial
scenarios, we assume that for all $i,j\in N$, $c_{i}<c_{j}+t_{ji}$, $\nu
_{i}<\nu _{j}+t_{ji}$, $r_{i}<r_{j}+t_{ji}$, and $t_{ij}<r_{j}-\nu _{i}$. By $\mathbf{X,}$ $\mathbf{D,}$  we denote vectors of quantities and random demands, 
and by $\mathbf{W}$ the $n\times n$ matrix of transshipped quantities respectively, for agents in $N$.  

The cooperative transshipment game is a cooperative game $(\dot{J},N)$ with the characteristic function $\dot{J}:2^{N}\rightarrow \mathbb{R},$ being a two-stage stochastic program with recourse, which assigns to any coalition $M\subseteq N$ the value $\dot{J}_{M}$  equal to
\begin{equation}
\dot{J}_{M}=\max_{\mathbf{X}}J_{M}(\mathbf{X})=\max_{\mathbf{X}}\mathsf{E}\left[
\sum_{i\in M}\left( {r_{i}\min (X_{i},D_{i})+\nu _{i}H_{i}-c_{i}X_{i}}%
\right) +R_{M}(\mathbf{X},\mathbf{D})\right]  \label{TheGame}
\end{equation}

where for given $\mathbf{X}$ and $\mathbf{D}$, 
\begin{eqnarray}  \label{G}
R_{M}(\mathbf{X},\mathbf{D})&=& \max_{\mathbf{W}} \sum_{i \in M}{\sum_{j \in M}{p_{ij}W_{ij}}} \\
& s.t.&  
 \sum_{j\in M}{W_{ij}} \leq H_{i}, \forall i \in M  \notag \\
& &\sum_{i\in M}{W_{ij}} \leq E_{j}, \forall j \in M  \notag \\
& &W_{ij} \geq 0, \forall i,j \in M  \notag.
\end{eqnarray}

The $H_{i}=\max (X_{i}-D_{i},0)$ and $E_{i}=\max(D_{i}-X_{i},0)$ are newsvendor $i$'s surplus, and unsatisfied demand respectively and $p_{ij}=r_{j}-\nu _{i}-t_{ij}$ is the marginal transshipment profit 
 from  $i$ to  $j$. Equation (\ref{TheGame}) shows that total expected profit in a transshipment coalition is consisting of sum of newsvendors' individual profits as well as the transshipment profit, i.e.\ $R_M$.

Let $\bm{\beta}=\{\beta _{i}|i \in N\}$ be the set of individual allocations in the coalition of $n$ newsvendors (grand coalition). The allocation $\bm{\beta}$ is in the core of the transshipment game if and only if $\sum_{i\in M}\beta _{i}\geq \dot{J}_{M}$ for all $M\subset N$, and $\sum_{i\in N}\beta _{i}=\dot{J}_{N}$\citep{Owen1995}. The key result of \citet{Slikker2005} ensures a non-empty core for a transshipment game without cooperation cost of size $n>1$.
This implies that  it is never disadvantageous for newsvendors to form ever larger coalitions. 
 
\subsection{Cooperative Transshipment Games with Identical Agents}

Clearly, identical agents must play symmetric games. In a symmetric cooperative game, the characteristic function is solely determined by the sizes of coalitions \citep{Luce1975}. For identical newsvendors with a unit transshipment between any two newsvendors results in the same profit $p$, which allows us to suppress the indices of $p_{ij}$. Therefore, a coalition can maximize its transshipment profit by carrying our transshipment in the way that there will be neither any surplus or shortage left. We have
$R_{n}(\mathbf{X},\mathbf{D})=p\min \left(\sum_{i=1}^{n}H_{i},\sum_{i=1}^{n}E_{i}\right).$ The expected transshipment amount is
\begin{equation}
\label{extran}
\omega_{n}(X)=\mathsf{E}\left[\min \left( \sum_{i=1}^{n}H_{i},\sum_{i=1}^{n}E_{i}\right)\right] =\mathsf{E}\left[\min \left(
\sum_{i=1}^{n}X_{i},\sum_{i=1}^{n}D_{i}\right) -\sum_{i=1}^{n}\min \left(
X_{i},D_{i}\right)\right].
\end{equation}%
The last equation holds by $\min (A,B)+C=\min (A+C,B+C)$. The vector of optimal quantities is essentially a singleton, therefore, $\mathbf{X}$ is replaced by a single variable $X$.
Furthermore, we have 
\begin{eqnarray}
&&\mathsf{E}\left[ \min (X,D)\right] =X-\int_{-\infty }^{X}F_{D}(\xi )d\xi \label{eq:swq1},
 \\
&&\mathsf{E}\left[ \min (nX,Z)\right] =nX-\int_{-\infty }^{nX}F_{Z}(\xi )d\xi \label{eq:swq2} 
\end{eqnarray}%
where $F_{D}$ ($f_{D}$) and $F_{Z}$ ($f_{Z}$) are cdfs (pdfs) of the random
variables $D$ and $Z=\sum_{i=1}^{n}D_{i}$ respectively. By substituting the terms in (\ref{TheGame}) and simplifying we obtain
\begin{equation}
J_{n}(X)=n(r-c)X-nt\int_{-\infty }^{X}F_{D}(\xi )d\xi -p\int_{-\infty}^{nX}F_{Z}(\xi )d\xi .  \label{Prof2}
\end{equation}%

For checking the non-emptiness of core in symmetric games, it is sufficient to check the core-membership of equal allocations for if a non-empty core exists, then it must contain equal allocations \citep{Shapley1967}. Therefore, the core of a symmetric transshipment game is non-empty if $m\beta_{n}\geq \dot{J}_{m}$, or $\beta_n \geq \beta_m$, for all $m\leq n$, where $\dot{J}_{m}=\max_{X}J_{m}(X)$ and $\beta_m=\dot{J}_{m}/m$.

\section{Expected Transshipments for Normally Distributed Demands} 

From now on we assume that demands at different locations are normally
distributed. The main motivation behind this assumption comes from the fact
that the normal distribution is a strictly \emph{stable} distribution %
\citep{Fristedt1997}; that is, for the symmetric case, the total demand $%
Z=\sum_{i=1}^{n}D_{i}$ is \emph{normally} distributed with $\mu _{Z}=n\mu $
and $\sigma _{Z}^{2}=n\left( 1+(n-1)\rho \right) \sigma ^{2}$ where $\rho $
is the correlation efficient between every pair of random variables\footnote{%
Note that in order for the covariance matrix to be positive-semidefinite, it
must be the case that $\frac{-1}{n-1}<\rho \leq 1$.}. \cite{Alfaro2003} show
that normal distribution is a good approximation of general distribution
functions in transshipment problem. \citet{Hartman2005}, and \citet{Dong2004}
also restrict their analysis to normal distributions when analyzing the
games among newsvendors. 

Let $\phi $ and $\Phi $ be the pdf and cdf of the standard normal distribution respectively. Using the transformation $Y=\left(X-\mu \right)/\sigma $ and letting $L_{n}=\sqrt{n/\left[ 1+(n-1)\rho \right] }$ we have
\begin{eqnarray}
\int_{-\infty }^{X}F_{D}(\xi )&=
&\sigma \int_{-\infty }^{Y}\Phi
\left( \xi \right) d\xi , \label{eq:drw1}\\
\int_{-\infty }^{nX}F_{Z}(\xi )&=
&n\sigma \int_{-\infty }^{Y}\Phi \left( L_{n}\xi \right) d\xi .\label{eq:drw2}
\end{eqnarray}%

\begin{theorem}
\label{Theo00}
For $Y>0$ $\mathsf{E}(\sum_{i=1}^{n}H_{i})>\mathsf{E}(\sum_{i=1}^{n}E_{i})$, for $Y<0$ $\mathsf{E}(\sum_{i=1}^{n}H_{i})<\mathsf{E}(\sum_{i=1}^{n}E_{i})$, and for $Y=0$ $\mathsf{E}(\sum_{i=1}^{n}H_{i})=\mathsf{E}(\sum_{i=1}^{n}E_{i})$.\end{theorem}

\begin{proof}
We have 
$\mathsf{E}(\sum_{i=1}^{n}H_{i})=n\mathsf{E}\left[ \max (Y-D,0)\right]
=n\int_{-\infty }^{Y}\Phi(\xi )d\xi$ , and
$\mathsf{E}(\sum_{i=1}^{n}E_{i})=n\mathsf{E}\left[ \max (D-Y,0)\right] =n(-Y+\int_{-\infty }^{Y}\Phi(\xi )d\xi) $
 Clearly, $Y>0$ implies $\mathsf{E}(\sum_{i=1}^{n}H_{i})>\mathsf{E}(\sum_{i=1}^{n}E_{i})$, $Y<0$ implies $\mathsf{E}(\sum_{i=1}^{n}H_{i})<\mathsf{E}(\sum_{i=1}^{n}E_{i})$, and $Y=0$ implies $\mathsf{E}(\sum_{i=1}^{n}H_{i})=\mathsf{E}(\sum_{i=1}^{n}E_{i})$.
\end{proof}
Therefore any coalition order quantity \emph{above} the mean results in a positive \emph{net} expected \emph{surplus}, and any coalition order quantity \emph{below} the mean results in a positive \emph{net} expected \emph{shortage}. Moreover, only the mean ensures a perfect match of expected shortage and surplus for the coalition.

\begin{theorem}
\label{Theo01}
A coalition's expected transshipment amount 
reaches its maximum at $Y=0$.
\end{theorem}

\begin{proof}
By equations (\ref{extran}), (\ref{eq:swq1}), and (\ref{eq:swq2}) we get 
$\omega_{n}(Y)=n\sigma \int_{-\infty }^{Y}\left[\Phi(\xi )-\Phi(L_{n}\xi )\right]d\xi. $
Thus, we have
$\frac{d\omega_{n} (Y)}{dY}=n\sigma\left( \Phi (Y)-\Phi \left( L_{n}Y\right) \right).
$ 
 The function $\Phi $ is convex below zero and concave above zero. As $L_{n}>0$, the expected transshipment is increasing
for $Y<0$, decreasing for $Y>0$, and reaches its maximum at $Y=0.$ 
\end{proof}

\section{Optimal Quantities}
\label{OQ}

Our main goal now is to characterize the optimal quantities for a
symmetric transshipment problem with $n$ newsvendors, or just a problem of
size $n$ for simplicity. Since $J_{n}(X)$ in (\ref{Prof2}) is concave on $X$%
, the optimal quantity can be found from the first order condition.
Let $X_{n}$ be the optimal quantity in the problem of size $n$ and  
$Y_{n}=(X_{n}-\mu )/\sigma $ as its normal transformation. 
The optimal quantity for a problem of size $n$ is obtained through 
\begin{equation}
R=\gamma \Phi (Y_{n})+\tilde{\gamma}\Phi \left( L_{n}Y_{n}\right) .  \label{OptY}
\end{equation}

where $R=(r-c)/(r-\nu)$ is the critical fractile, $\gamma =t/(r-\nu)$, and $\tilde{\gamma}=1-\gamma = p/(r-\nu).$ We use the modified notation to
take advantage of the symmetry in the problem.  The main
challenge in characterizing the optimal quantity $Y_{n}$ is its \emph{%
implicit} form given in (\ref{OptY}). The following lemma shows the relation between optimal quantities in a problem of size $n$ and that of single constructing newsvendor. Due to the symmetry in the problem, we construct the proofs for only half of the cases.

\begin{lemma}
\label{LemY1}For $n\geq 1$, if $Y_{1}>0$, then $Y_{n}>0$; if $Y_{1}<0$, then 
$Y_{n}<0$; and if $Y_{1}=0$, then $Y_{n}=0$.
\end{lemma}

\begin{proof}
Take the case with $Y_{1} <0$ which is equivalent to $R<1/2.$ By (\ref{OptY}) we get $R=\gamma \Phi (Y_{n})+\left( 1-\gamma \right)  \Phi\left(L_{n}Y_{n}\right) <1/2$. Since $Y_{n}$ and $L_{n}Y_{n}$ have the same sign, applying the monotonic increasing function $\Phi$ results in both terms simultaneously being (a) less than, (b) greater than, or (c) equal to $1/2.$ Noting $ 0\leq \gamma <1,$ it directly follows that the cases (b) and (c) result in contradiction. 
\end{proof}

Let $g=r-c$ and $\tilde{g}=c-{\nu }$ be the benefits of selling a unit
product and avoiding salvage markdown respectively. If $R<1/2$, i.e.\ $g<%
\tilde{g}$, then the optimal quantity for an individual newsvendor in \emph{any} problem is less than the demand mean $\mu $, hence we refer to this type of newsvendor (problem) as an \emph{under-mean newsvendor} (problem). Similarly, if $R>1/2$, i.e.\ $g>\tilde{g}$, then the optimal quantity for an individual newsvendor in any problem is larger than the demand mean $\mu $, hence we refer to this type of newsvendor (problem) as an \emph{over-mean newsvendor} (problem). Finally, if $R=1/2$, i.e.\ $g=\tilde{g}$, then the optimal quantity for an individual newsvendor in any problem equals the demand mean $\mu $, hence we refer to this type of newsvendor (problem) as a \emph{mean newsvendor} (problem). Lemma 1 then states that irrespective of transportation cost, any coalition of identical over-mean (under-mean) newsvendors will remain over-mean (under-mean).
The following theorem shows that the over-mean problems reduce their optimal
quantities, and the under-mean problems increase their optimal quantities as
their sizes grow.

\begin{theorem}
\label{TheoY2} For over-mean problems, $Y_{1}>Y_{2}>...>Y_{n}>...>0$. For
under-mean problems, $Y_{1}<Y_{2}<...<Y_{n}<...<0$.
\end{theorem}

\begin{proof}
Take that case with $Y_{1}<0$ and suppose that  $Y_{n-1}\geq Y_{n}$ for some $n\geq 2$. The function $\Phi$ is
monotonic increasing hence we have $\Phi (Y_{n-1})\geq \Phi (Y_{n})$. By Lemma 1, $Y_{n}<0$ for all $n\geq 1$, thus we also get  $L_{n-1}Y_{n-1}>L_{n}Y_{n}$. Applying the $\Phi$ function will also keep the direction of inequality. Hence we get $\gamma \Phi (Y_{n-1})+\left( 1-\gamma\right) \Phi \left( L_{n-1}Y_{n-1}\right) >\gamma \Phi (Y_{n})+\left( 1-\gamma\right) \Phi \left( L_{n}Y_{n}\right)$
which violates the optimality condition in (\ref{OptY}).
Therefore, $Y_{n-1}<Y_{n}$ for all $n\geq 2$. 
The cases for $Y_{1}>0$ and $Y_{1}=0$ are proven in a similar manner. 
\end{proof}
In conjunction with Theorem \ref{Theo01}, Theorem \ref{TheoY2} reveals that as the size of transshipment coalitions grows, coalitions increase their expected transshipment amount.
Although the risk pooling mechanism naturally embedded in a transshipping
coalition---revealed in Theorem \ref{TheoY2}---makes the mean $\mu $ a
natural target for the optimal quantity in a coalition, this optimal
quantity does \emph{not} necessarily converge to the mean $\mu $ as the
problem size grows. This is shown in Theorem \ref{lim} presented later in
this section. Before presenting this theorem we first get a closer look at
the sequence $L_{n}Y_{n}$ which is the other ingredient of implicit formula (\ref%
{OptY}).

\begin{theorem}
\label{Theo2fh} For over-mean problems, $0<Y_{1}<L_{2}Y_{2}<...<L_{n}Y_{n}<...$. For
under-mean problems, $0>Y_{1}>L_{2}Y_{2}>...>L_{n}Y_{n}>...$.
\end{theorem}

\begin{proof}
By (\ref{OptY}), for arbitrary $n$ we have%
\begin{equation}
\gamma \left[ \Phi (Y_{n})-\Phi (Y_{n-1})\right] =\left( 1-\gamma \right) \left[ \Phi \left( L_{n-1}Y_{n-1}\right) -\Phi \left( L_{n}Y_{n}\right) \right].  \label{EQP1} \end{equation}
 For the case with $Y_{1}<0$, the sequence $Y_{n}$ is
increasing and negative (by Lemma 1 and Theorem 1). Hence $\Phi (Y_{n})-\Phi
(Y_{n-1})<0.$ In order for the equality to hold in (\ref{EQP1}), we must
have $\Phi \left( L_{n-1}Y_{n-1}\right) -\Phi \left( L_{n}Y_{n}\right)
<0$ as well. The monotone increasing property of $\Phi $ requires that $%
L_{n}Y_{n}>L_{n-1}Y_{n-1}$. 
\end{proof}

Theorem \ref{TheoY2} and Theorem \ref{Theo2fh} show a \emph{complementary}
behavior of the sequences $Y_{n}$ and $L_{n}Y_{n}$; whenever one of them is \emph{%
descending} the other must be \emph{ascending}. This must be so in order to
satisfy the equation (\ref{OptY}). We are now ready to present the main
result of this section.

\begin{theorem}
\label{lim} Let $Y_{\infty}=\lim_{n\rightarrow \infty }Y_{n}$ and $%
L_{\infty}Y_{\infty}=\lim_{n\rightarrow \infty }L_{n}Y_{n}$. The following statements are true:

\begin{table}[htbp]
\centering
\begin{tabular}{cc|cc}
Game Type & Cut & $\Phi (Y_{\infty})$ & $\Phi (L_{\infty}Y_{\infty})$ \\ \hline\hline
\multirow{2}{*}{Over-mean} & $t/2<\tilde{g}$ & $1/2$ & $1-\left( \tilde{g}%
-t/2\right) /p$ \\ 
& $t/2\geq \tilde{g}$ & $1-\tilde{g}/t$ & $1$ \\ \hline
\multirow{2}{*}{Under-mean} & $t/2<g$ & $1/2$ & $\left( g-t/2\right) /p$ \\ 
& $t/2\geq g$ & $g/t$ & $0$  
\end{tabular}
\end{table}
\end{theorem}

\begin{proof}
We need only consider the under-mean case where we have 
$g<\tilde{g}$. 
The sequence $\{Y_{n}\}$ is monotonic increasing and bounded above by $0$, so it
converges to some finite limit. Hence $\{\Phi (Y_{n})\}$ also is a
monotonically increasing sequence bounded above by $1/2$, so that $0<\Phi
(Y_{\infty})\leq 1/2$. Thus, $\{\Phi (L_{n}Y_{n})\}$ must be a monotonically decreasing
sequence bounded below by $0$, so that $0\leq \Phi (L_{\infty}Y_{\infty})<1/2$.

There are only two possible scenarios for $%
\Phi (Y_{\infty})$ and $\Phi (L_{\infty}Y_{\infty})$: $\{\Phi (Y_{\infty})=1/2$ and $\Phi (L_{\infty}Y_{\infty})>0\}$, or $\{\Phi
(Y_{\infty})\leq 1/2$ and $\Phi (L_{\infty}Y_{\infty})=0\}$ (the case of $\{\Phi (Y_{\infty})\leq 1/2$ and $\Phi
(L_{\infty}Y_{\infty})>0\}$ is impossible). 

Suppose $\Phi (L_{\infty}Y_{\infty})=0$, the second case. Then from equation (\ref{OptY}) we have $R=\gamma \Phi (Y_{\infty}) \leq 1/2
$ so $g \leq t/2$ and $t \geq 2g$. In fact,  $\Phi (Y_{\infty})= R/\gamma = g/t$.
Now suppose $\Phi (Y_{\infty})=1/2$ for the first case. From equation (\ref{OptY}) we have $R = \gamma/2 + (1-\gamma)\Phi (L_{\infty}Y_{\infty})$
so that $g = t/2 + p\Phi (L_{\infty}Y_{\infty})$ and $\Phi (L_{\infty}Y_{\infty}) = (g-t/2)/p$. But for this to be positive we must have $g>t/2$ so that $t<2g$.
\end{proof}
Figure \ref{fig:quant} shows the $Y_{\infty}$ as a function of $t$ for over-mean and under-mean problems.
\begin{figure}[b] 
\centering
\subfloat[Over-mean
Games]{\includegraphics[width=0.50\textwidth]{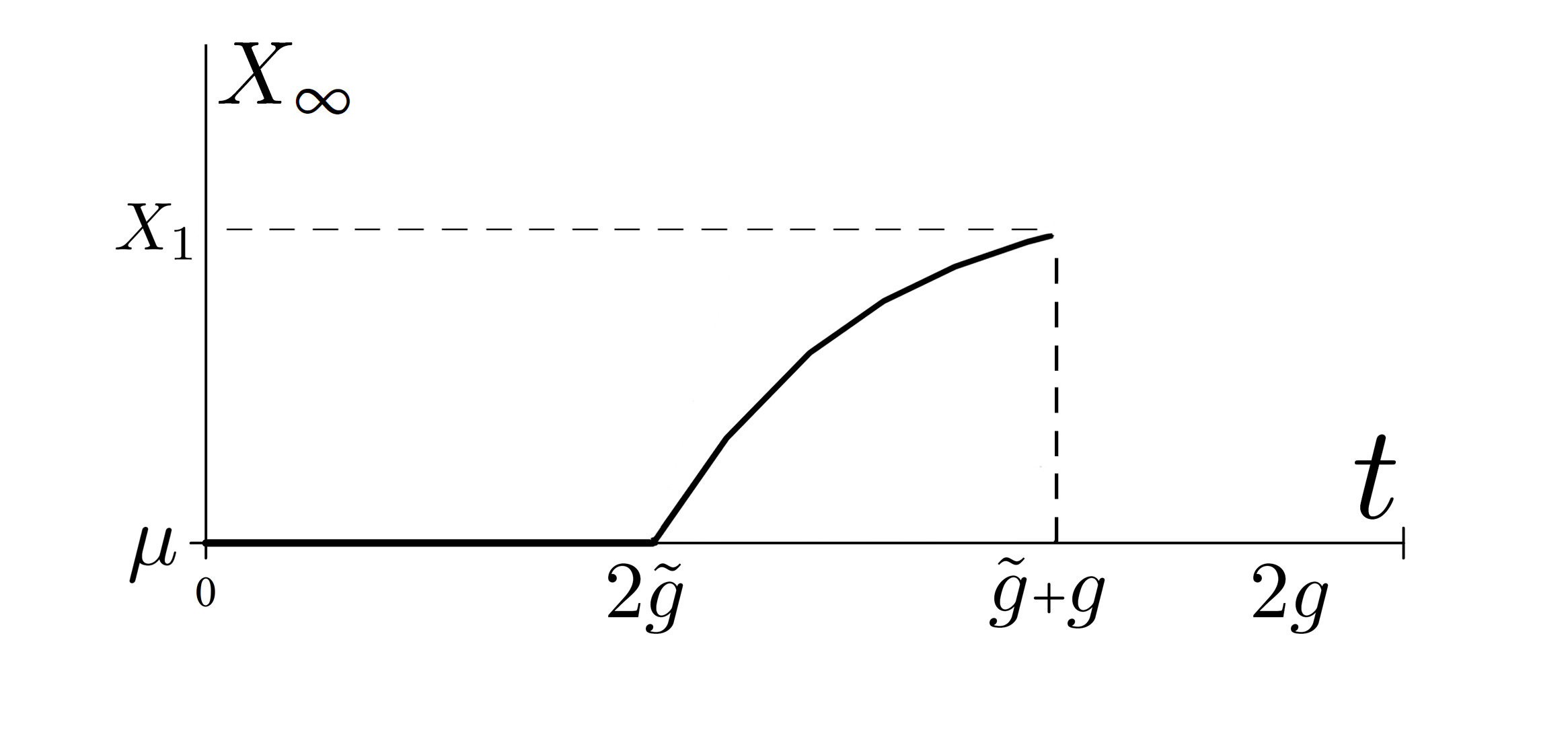}}  
\subfloat[Under-mean
Games]{\includegraphics[width=0.50\textwidth]{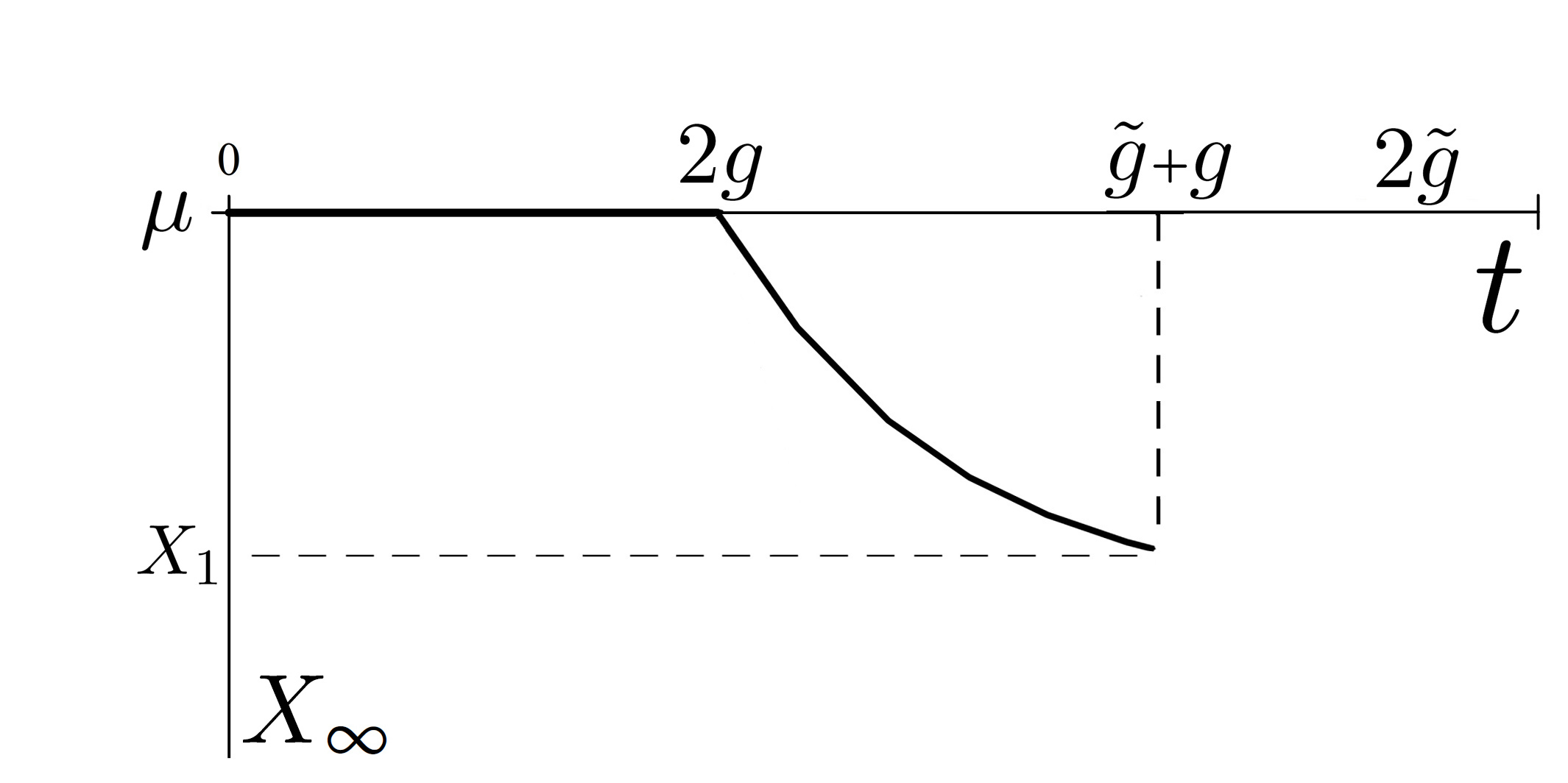}}
\caption{$Y_{\infty}$ as a function of $t$.}
\label{fig:quant}
\end{figure}

\vskip 1ex

\subsection{Interpretation of the Cut Values}
The cut value in Theorem \ref{lim} has an interesting interpretation. 
When transshipment occurs, the cost of transport is $t$ per unit, and equal division of the cost allocates $t/2$ to the sender and the receiver. 

For the under-mean case we have $t/2<\tilde{g}$; thus, this share of transport cost is less than the cost of markdown for senders, who are \emph{always} better off by transshipping. Now, the incentives at the receivers' side depend on their share of transport cost, $t/2$ as well. In the case where the receiver's share of transport cost $t/2$ is less than benefit on a sale $g$, they are willing to do transshipment as well. Therefore, the coalition could ideally order a quantity that results in a perfect match of total expected shortage and surplus. This, by Theorem \ref{Theo00} and \ref{Theo01}, would happen at the mean. However, since we deal with under mean games there always is some positive net expected shortage in the game since $Y_n<0$ for each $n$; thus, the optimal order size being equal, the mean can only be the limit of these order sizes. In the case where the receiver's share of transport cost $t/2$ is higher than benefit on a sale $g$, the receivers do not see transshipment as a totally desirable option. Therefore the conflicting incentives of senders and receivers stop the coalition order quantity at $Y_{\infty }=\Phi ^{-1}(g/t)$, which is short of the mean. The net shortage $S_n$ in the under-mean game of size $n$ equals $\sum_{i=1}^{n}E_{i}- \sum_{i=1}^{n}H_{i}= \sum_{i=1}^{n}D_{i}-nX_n$. This shortage occurs after all transshipments have been done. This shortage needs to be equally shared by all newsvendors, and the newsvendor's share is $S_n/n= (\sum_{i=1}^{n}D_{i})/n-X_n$. 
This share is normally distributed with mean $\mu-X_n=-\sigma Y_n$ and standard deviation $\sigma$. Therefore, $P(S_n/n\leq 0)=\Phi(Y_n)$, and the probability $P(S_n\leq 0)=\Phi(Y_n)$ that the shortage does \emph{not} occur tends to $g/t$ as $n$ increases for under-mean games.

For over mean games $t/2 < g$; thus, this share of transport cost is less than a sale for a receiver, who is \emph{always} better off by accepting transshipment. 
In the case where the senders' share of transport cost $t/2$ is less than markdown cost $\tilde{g}$, they too would rather increase the chances of transshipment. Therefore, the coalition could ideally order a quantity that results in a perfect match of total expected shortage and surplus. 
However, since we deal with over mean games there always is some positive net expected surplus in the game since $Y_n>0$ for each $n$; thus, the optimal order size equal to the mean can only be the limit of these order sizes. 
In the case that $t/2> \tilde{g}$, the share of transport cost is more than the cost of marking down for the senders. Thus, the potential sender may prefer markdown over transshipment. Therefore the conflicting incentives of senders and receivers stops the coalition order quantity at $Y_{\infty }=\Phi ^{-1}(t-\tilde{g})/t$ which is above the mean. Therefore, the probability $P(S_n> 0)=\Phi(Y_n)$ that the shortage \emph{does} occur tends to $\tilde{g}/t$ as $n$ increases for over-mean games.

\section{Formula for Equal Core Allocations}

\label{EP} We now derive a closed-form formula for the maximum expected
profit $\dot{J}_{n}$ in symmetric transshipment with normally distributed
demands. Start from equation (\ref{Prof2}); through standardization, 
changing the integral arguments, and integration by
parts we get 
\begin{equation}
J_{n}(X)=n\left( g+\tilde{g}\right) \left( R(\mu +\sigma Y)-\gamma \sigma %
\left[ \phi (Y)+Y\Phi (Y)\right] -\tilde{\gamma}\sigma \left[ \phi \left(
L_{n}Y\right) /L_{n}+Y\Phi \left( L_{n}Y\right) \right] \right) 
\end{equation}%
For optimal order quantities, $Y_{n}$, applying the optimality conditions in
(\ref{OptY}) obtains the following closed form expression: 
\begin{equation}
\dot{J}_{n}=n\left( g+\tilde{g}\right) \left( R\mu -\sigma \left[ \gamma
\phi (Y_{n})+\tilde{\gamma}\phi \left( L_{n}Y_{n}\right) /L_{n}\right]
\right).  \label{EqnL}
\end{equation}
Thus the formula for equal core allocation is as follows
\begin{equation}
\beta_{n}=\left( g+\tilde{g}\right) \left( R\mu -\sigma \left[ \gamma
\phi (Y_{n})+\tilde{\gamma}\phi \left( L_{n}Y_{n}\right) /L_{n}\right]
\right).  \label{EqnL2}
\end{equation}
Equation (\ref{EqnL}) and (\ref{EqnL2}) do not guarantee that the $\dot{J}_{n}$ and $\beta_{n}$ are always non-negative.
 This is due to the fact that under normal distribution with relatively large standard deviations, negative market demands are non-negligible \citep{Hartman2005}. In order to avoid this, it suffices to assume that $\sigma/\mu \leq g/\left[(g+\tilde{g})\phi\left(\Phi ^{-1}\left( R\right)\right) \right]$.

\paragraph{Acknowledgments}
This research has been supported by the Natural Sciences and Engineering
Research Council of Canada (NSERC) Grant OPG0105675. Moreover, the research
of Behzad Hezarkhani has also been supported by the Canadian Purchasing
Research Foundation. The authors are greatly indebted for these supports.

\end{document}